\documentclass[10pt, reqno]{article}
\usepackage{amsmath,amssymb,amsfonts,amsthm,amstext,amsxtra}
\usepackage{xypic}
\usepackage{color}
\usepackage[dvips]{epsfig,graphics}
\usepackage[bookmarksnumbered, colorlinks, plainpages]{hyperref}
\usepackage[all,knot,arc,poly]{xy}
\newtheorem{thm}{Theorem}[section]
\newtheorem{lem}[thm]{Lemma}

\newtheorem{cor}[thm]{Corollary}
\theoremstyle{definition}
\newtheorem{defn}[thm]{Definition}

\theoremstyle{remark}
\newtheorem{rem}[thm]{Remark}

\theoremstyle{fact}

\numberwithin{equation}{section}
\begin{document}
\begin{center}
{\Large \textbf{\\Standard Projective Simplicial Kernels and the Second Abelian Cohomology of Topological Groups }}\\[5mm]
{\large \textbf{H. Sahleh\footnote{\emph{E-mail: sahleh@guilan.ac.ir $^\mathrm{2}$E-mail: h.e.koshkoshi@guilan.ac.ir}}  and H.E. Koshkoshi$^\mathrm{2}$}}\\[1mm]
$^\mathrm{1}${\footnotesize \it Department of  Mathematics, Faculty of Mathematical Sciences, University of Guilan\\ P. O. Box 1914, Rasht, Iran}\\[1mm]
$^\mathrm{2}${\footnotesize \it Department of  Mathematics, Faculty of Mathematical Sciences, University of Guilan}
\end{center}
\date{}
\newcommand{\stk}[1]{\stackrel{#1}{\longrightarrow}}
\begin{abstract}
Let $A$ be an abelian topological $G$-module. We give an interpretion for the second cohomology, $H^{2}(G,A)$, of $G$ with coefficients in $A$. As a result we show that if $P$ is a projective topological group, then $H^{2}(P,A)=0$ for every abelian topological $P$-module $A$.
\end{abstract}
\section{Introduction}
Let $G$  be a topological group and $A$ an abelian topological group. It is said that $A$ is an abelian topological $G$-module, whenever $G$   acts continuously on $A$. For all $g\in G$ and $a\in A$ we denote the action of $g$ on $a$ by $^{g}a$. Suppose that $A$ is an abelian topological $G$-module. Hu \cite{Hu} showed that if $G$ is Hausdorff, then there is a one to one correspondence between the second cohomology,  $H^{2}(G,A)$, of $G$ with coefficients in $A$ and the set, $Ext_{s}(G,A)$, of all equivalence classes of topological extensions with continuous sections. Thus, $H^{2}(G,A)$ induces  a group structure on $Ext_{s}(G,A)$ and consequently, under this group product, $H^{2}(G,A)$ is isomorphic to  $Ext_{s}(G,A)$. It is known that $Ext_{s}(G,A)$ with Baer sum is an abelian group and the Baer sum on $Ext_{s}(G,A)$ coincides with the group product induced by $H^{2}(G,A)$. In other words, if $G$ is Hausdorff then, $H^{2}(G,A)$ is isomorphic to $Ext_{s}(G,A)$ with Baer sum \cite{Al}. Similarly, we define $Opext_{s}(G,A)$ and we conclude the same result without Hausdorffness of $G$, i.e., there is an isomorphism between $H^{2}(G,A)$ and $Opext_{s}(G,A)$. By using the concept of projective topological group we introduce the notion of a standard simplicial kernel of a topological group $G$.
\par in section \ref{section 2}, we prove that every Markov (Graev) topological group is projective and also we show that projectivity in the category of topological groups is equivalent to $\mathfrak{F}$-projectivity in the sense of Hall \cite{Hal}. In addition, we introduce the notion of a standard simplcial kernel of a topological group $G$. Finally, we define  $Opext_{s}(G,A)$ and  we conclude that $H^{2}(G,A)\cong Opext_{s}(G,A)$, with Baer sum.
\par In section \ref{section 3}, we give a characterization of $H^{2}(G,A)$, when $A$ is an abelian topological $G$-module. As a result, we show that if $P$ is a projective topological group then $H^{2}(P,A)=0$, for every abelain topological $P$-module. 
\section{ Standard Projective Simplicial Kernels.}\label{section 2}
In this section, we introduce the notion of a standard projective simplicial kernel of a topological group $G$.
\begin{defn}\label{Definition 1.1.} A topological group $P$ is said to be a projective topological group if, for every continuous epimorphism $\pi:A\rightarrow B$, where $\pi$ has a continuous section, and for every continuous homomorphism $f:P\rightarrow B$, there exists a continuous homomorphism $g:P\rightarrow A$ such that the following diagram is commutative
$$\xymatrix{ & P\ar[d]^{f} \ar[dl]_{g} & \\ A\ar[r]^{\pi}&B\ar[r]&1.}$$
\end{defn}
If $X$ is a topological space, then the (Markov) free topological group over $X$ is the pair
$(F_{X},\sigma_{X}:X\to F_{X})$ in which $F_{X}$ is a group equipped with the finest group topology such that $\sigma_{X}$ is continuous. Such a topology always exists \cite{Tho}, and has the following universal
property: every continuous map $f$ from $X$ to an
arbitrary topological group $G$ lifts to a unique continuous homomorphism $\bar f$, .i.e., $\bar f \circ\sigma_{X}=f$. Similarly,  one can define the Graev free topological group, $(F^{*}_{X},\sigma_{X}^{*})$, over a pointed topological space $(X,e)$. For information on free (abelian) topological groups see \cite{Gra, Mar, Tho}.
The following facts about Markov (Graev) free toplogiacal group are well-known:
\emph{\begin{itemize}
  \item[(1)] The Markov (Graev) free topological group over a (pointed) topological
space $X$ ($(X,e)$) is the free group with the same canonical map over the (pointed) set
$X$ ($(X,e)$);
  \item[(2)] $F_{X}$ ($F^{*}_{X}$) is Hausdorff if and only if $X$ is functionally Hausdorff;
  \item[(3)] $\sigma_{X}$ ($\sigma_{X}^{*}$) is a homeomorphic embedding if and only if $X$ is  completely regular;
  \item[(4)] $\sigma_{X}$  ($\sigma_{X}^{*}$) is a closed homeomorphic embedding if and only if $X$ is Tychonov.
\end{itemize}}
The elements of  free group, $F_{X}$, over a set $X$ is denoted by $|x_{1}|^{\epsilon_{1}}...|x_{n}|^{\epsilon_{n}}$, where $\epsilon_{i}=\pm 1$. This notation is useful whenever $X$ is a group.
\begin{lem}\label{Lemma 1.2.} Every Markov (Graev) free topological group is a projective topological group.
\end{lem}
\begin{proof} Assume that $F$ is a Markov free topological group over the space $X$, and let $\pi:A\rightarrow B$ be a continuous epimorphism having a continuous section $s$ and $f:F\rightarrow B$ a continuous homomorphism. Then, there is a unique continuous homomorphism $\bar{f}:F\rightarrow A$ such that the following diagram is commutative
$$\xymatrix{ X\ar[r]^{\sigma_{X}}\ar[d]^{\sigma_{X}} & F\ar[r]^{f} & B\ar[d]^{s} \\
F\ar[rr]^{\bar f} & & A}$$
Hence, $sf=\bar{f}$ on $X$. So, $\pi\bar{f}=\pi(sf)=(\pi s)f=f$ on $X$. Since $X$ generates the group $F$, then it is easy to see that $\pi\bar{f}=f$. Consequently, $F$ is a projective topological group.
\par Now, Let $X$ be a topological space with a fixed point $e\in X$. Assume that $F^{*}$ is the free topological group in the sense of Graev over $(X,e)$, and let $\pi:A\rightarrow B$ be a continuous epimorphism with continuous section $s:B\rightarrow A$. We can assume that $s$ is a normal section, i.e., $s(1)=1$, since it is enough to define the new continuous section $\tilde{s}:B\rightarrow A$ by $\tilde{s}(b)=s(b)s(1)^{-1}$.  In fact, since $e$ is the neutral element of $F^{*}$, then $f(e)=1$. Thus, $sf(e)=1$. Hence there exists a unique continuous homomorphism $\tilde{f}:F^{*}\rightarrow A$ such that $sf=\tilde{f}$ on $X$. The rest of the proof is the same as the first part.
\end{proof}
Note that the free functor $F_{(-)}$ from the category, $\mathcal{T}$, of topological spaces to the category, $\mathcal{TG}$, of topological groups is left adjoint to the forgetful functor, $G:\mathcal{TG} \to \mathcal{T}$. Let $\mathfrak{F}$ be the class of all continuous epimorphisms having a continuous section. Thus, by \cite[Theorem 2]{Hal} one can see the following:
\begin{rem}\label{Remark 1.3.} A topological group $P$ is $\mathfrak{F}$-projective  if and only if $P$ is projective.
\end{rem}
\begin{defn}\label{Definition 1.4.} Let $M$, $F$ and $G$ be topological groups, and let $\iota_{0}, \iota_{1}:M\rightarrow F$ and $\tau:F\rightarrow G$ be continuous homomorphisms. It is said that $(M,\iota_{0},\iota_{1})$ is a simplicial kernel of $\tau$, whenever $\tau\iota_{0}=\tau\iota_{1}$
and it has the following universal property:
 \par if $\jmath_{0}, \jmath_{1}:K \rightarrow F$ are continuous homomorphisms and  $\pi \jmath_{0}=\pi \jmath_{1}$, then there exists a unique continuous homomorphism  $h:K\rightarrow M$ such that $\jmath_{0}=\iota_{0}h$ and $\jmath_{1}=\iota_{1}h$, i.e., $h$ commutes the  following diagram.
\begin{center}
$\xymatrix{ K \ar @{->} @< 3pt> [r]^{\jmath_{0}}
\ar @{ ->} @<-3pt> [r]_{\jmath_{1}} \ar [d]_{h}& F\ar @{=}[d]\ar [r]^{\tau}& G \\M \ar @{->} @< 3pt> [r]^{\iota_{0}}
\ar @{ ->} @<-3pt> [r]_{\iota_{1}} & F\ar [r]^{\tau}& G }$
\end{center}
\end{defn}
\begin{defn}\label{Definition 1.5.}  Let $(M,\iota_{0},\iota_{1})$ be a simplicial kernel of $\tau:P\rightarrow G$. The quadruplet $(M,\iota_{0},\iota_{1},\tau)$ is said to be  a \emph{projective simplicial kernel of $G$}, if $P$ is a projective topological group and $\tau$ is a continuous (open) epimorphism  which has a continuous section $s:G\rightarrow P$. If $P$ is a Markov (Graev) free topological group, then $(M,\iota_{0},\iota_{1},\tau)$ is called a \emph{Markov (Graev) simplicial kernel of $G$}.
\end{defn}
\begin{lem}\label{Lemma 1.6.} If a continuous homomorphism $f:G\to H$ has a continuous section, then $f$ is an open epimorphism.
\end{lem}
\begin{proof}
 The proof is a standard argument. 
 \end{proof}
For any topological group $G$ there exists at least one projective simplicial kernel of $G$. Since the identity map $Id_{G}:G\rightarrow G$ lifts to a unique homomorphism $\tau_{G}:F_{G}\to G$. Thus, $\tau_{G}\sigma_{G}=Id_{G}$.  The uniqueness property implies that $\tau_{G}$ is a map by the following rule:
 \begin{center}
 $\tau_{G}:F_{G}\rightarrow G$, $|g_{1}|^{\epsilon_{1}}...|g_{n}|^{\epsilon_{n}}\mapsto g_{1}^{\epsilon_{1}}...g_{n}^{\epsilon_{n}}$ where $\epsilon_{i}=\pm1$.
\end{center}
The map $\tau_{G}$ is called \emph{multiplication map}. Obviously, $\sigma_{G}:G\rightarrow F_{G}$, $g\mapsto |g|$, is a continuous section for $\tau_{G}$.  By Lemma \ref{Lemma 1.6.}, $\tau_{G}$ is an open continuous epimorphism.
 We take $M_{G}=\{(x,y)|x,y\in F_{G}, \tau_{G}(x)=\tau_{G}(y)\}\subset F_{G}\times F_{G}$. Hence, $M_{G}$ has subspace topology induced by the product topology $F_{G}\times F_{G}$. Suppose $\iota_{0}=\iota_{0}^{G},\iota_{1}=\imath_{1}^{G}:M_{G}\rightarrow F_{G}$ are the canonical projection maps, where $\iota_{0}^{G}(x,y)=x$ and $\iota_{1}^{G}(x,y)=y$. It is easy to see that $(M_{G},\iota_{0}^{G},\iota_{1}^{G})$ is a simplicial kernel of $\tau_{G}$. Hence,  $(M_{G},\iota_{0}^{G},\iota_{1}^{G},\tau_{G})$ is a projective simplicial kernel of $G$ and we call it the \emph{standard Markov simplicial kernel of }$G$. By a similar way we can define the \emph{standard Graev (projective) simplicial kernel of }$G$.
\\
\par Suppose that $(M,\iota_{0},\iota_{1},\tau)$ is a  projective simplicial kernel of $G$. We sometimes denoted it by the following $$\xymatrix{M \ar @{->} @< 3pt> [r]^{\iota_{0}}
\ar @{ ->} @<-3pt> [r]_{\iota_{1}} & P\ar [r]^{\tau}& G \ar [r]&1}. \eqno(*)$$
\par If $(M_{G},\iota_{0},\iota_{1},\tau_{G})$ is a standard projective simplicial kernel of $G$, then put $\Delta_{P}=\{(x,x)|x\in P\}\subset M$. Now, let $A$ be an abelian topological $G$-module. Obviously, $A$ is a topological $P$-module via $\tau$ and a topological $M$-module via $\tau\iota_{0}$ (or $\tau\iota_{1}$).
Define $Z^{1}(M,A)=\{\alpha|\alpha\in Der_{c}(M,A), \alpha(\Delta_{P})=0\}$.  It is clear that $Z^{1}(M,A)$ is a subgroup of abelian group $Der_{c}(M,A)$.
\par Consider
\begin{center}
$\eta:Der_{c}(P,A)\rightarrow Z^{1}(M,A)$\\  $\alpha\mapsto \alpha\iota_{0}\alpha\iota_{1}^{-1}$
\end{center}
 Obviously, $\eta$ is well-defined and since $A$ is an abelian group, then
\begin{center}
$\eta(\alpha.\beta)=(\alpha.\beta\iota_{0})(\alpha.\beta\iota_{1}^{-1})=(\alpha\iota_{0}.\beta\iota_{0})(\alpha\iota_{1}^{-1}.\beta\iota_{1}^{-1})$ \\
$=(\alpha\iota_{0}\alpha\iota_{1}^{-1})(\beta\iota_{0}\beta\iota_{1}^{-1})=\eta(\alpha)\eta(\beta)$.
\end{center}
Therefore, $\eta$ is a homomorphism.
 Clearly, $\eta$ induces a congruence equivalence relation $\sim$ in $Z^{1}(M,A)$. In other words, for  $\alpha,\beta\in Z^{1}(M,A)$, the equivalence relation $\sim$  is defined as follows:
 \begin{center}
 $\alpha \sim \beta$ whenever there is $\gamma\in Der_{c}(P,A)$ and $\beta=\gamma\iota_{0}\gamma\iota_{1}^{-1}\alpha$.
 \end{center}
\begin{defn}\label{Definition 1.7.} It is said that a short exact sequence $$e: 1\to A\stackrel{\chi}\to E\stackrel{\pi}\to G\to 1$$ is  a proper extension of $G$ by $A$ whenever $\chi$ is a homeomorphic embedding and $\pi$ is an open continuous homomorphism. By a section for $e$ we mean a continuous map $s:G\to E$ such that $\pi s=Id_{G}$.
\end{defn}
Let $e$ be a  proper extension of $G$ by $A$ which has a section $s:G\to E$. In addition, assume that $A$ is abelian. Then the proper
extension $e$ gives rise to a topological $G$-module structure on $A$, which is well-defined
by
\begin{center}
$^{g}a=\chi^{-1}(s(g)\chi(a)s(g)^{-1})$ , where $g\in G$ , $a\in A$.
\end{center}
 One can easily see that for each $g\in G$ and $a\in A$, the element $^{g}a$ dose not depend on the choice of  continuous sections.
\par Now, let $A$ be an abelian topological $G$-module. We denote by $opext_{s}(G,A)$ the set of all proper extensions of $G$ by $A$ having continuous sections and corresponding to the given way in which $G$ acts on $A$. We define an equivalence relation $\equiv$ on $opext_{s}(G,A)$ as follows:
\par
 Let $e_{i}: 1\to A\stackrel{\chi_{i}}\to E\stackrel{\pi_{i}}\to G\to 1\in opext_{s}(G,A)$ for $i=0,1$.
 $e_{0}\equiv e_{1}$ whenever there is a continuous (open) homomorphism $\theta:E_{0}\to E_{1}$ so that the following diagram is commutative.
 $$\xymatrix{1\ar[r]& A\ar[r]^{\chi_{0}}\ar@{=}[d]& E_{0}\ar[r]^{\pi_{0}}\ar[d]^{\theta}& G\ar[r]\ar@{=}[d]&1
 \\ 1\ar[r]& A\ar[r]^{\chi_{1}}& E_{1}\ar[r]^{\pi_{1}}& G\ar[r]&1}$$
 We denote the set of all equivalence classes by $Opext_{s}(G,A)$. Note that if $G$ is Hausdorff, then $Opext_{s}(G,A)=Ext_{s}(G,A)$. One can verify  by a similar argument as in \cite{Al, Hu} that we have the following.
\begin{thm}\label{Theorem 1.8.}   Let $A$ be an abelian topological $G$-module, then $H^{2}(G,A)$ in the sense of Hu, is  isomorphic to $Opext_{s}(G,A)$ with Baer sum.
\end{thm}
\section{the Main Theorem }\label{section 3}
In this section we will prove the  main theorem. Sometimes we denote the inclusion maps and epics by the arrows $\hookrightarrow$ and $\twoheadrightarrow$, respectively. Let $A$ be a topological $G$-module. Thus, under the action $(g,a)\stackrel{\phi}\mapsto \ ^{g}a$ we may consider the topological semidirect product $A\rtimes _{\phi} G$ (see \cite[Section 6]{Sah1}).
\begin{thm}\label{Theorem 2.1.} Let $A$ be an abelian topological $G$-module and let $(M,\iota_{0},\iota_{1},\tau)$ be a standard projective simplicial kernel of $G$. Then,  $H^{2}(G,A)$ is canonically isomorphic to $Z^{1}(M,A)/\sim$.
\end{thm}
\begin{proof} By Theorem \ref{Theorem 1.8.}, it is enough to show that $Opext_{s}(G,A)$ is isomorphic to $Z^{1}(M,A)/\sim$.
\par Let $\alpha\in Z^{1}(M,A)$. Since $A$ is a $G$-module, then $A$ is a $P$-module via $\tau$. Thus, we may consider topological semidirect product $A\rtimes P$. Now, we define a relation  $\sim$ in the semidirect product $A\rtimes P$, as follows
\begin{center}
$(a,x)\sim (b,y)$ if and only if  $\tau(x)=\tau(y)$ and $a\alpha(x,y)=b$.
\end{center}
Obviously, $\sim$ is reflexive. Let  $x,y$ and $z$ be arbitrary elements of $P$ such that $\tau(x)=\tau(y)=\tau(z)$. Then
$$\alpha(x,y)\alpha(y,z)=\alpha(x,z)  \eqno(3.1)$$
since,
\begin{center}
$\alpha(x,z)=\alpha((x,y)(1,y^{-1}z))=\alpha(x,y).^{y}\alpha(1,y^{-1}z)=\alpha(x,y).^{y}\alpha(y^{-1}y,y^{-1}z)$
$$=\alpha(x,y).^{y}\alpha((y^{-1},y^{-1})(y,z))
=\alpha(x,y).^{y}(\alpha(y^{-1},y^{-1}).^{y^{-1}}\alpha(y,z))$$
$=\alpha(x,y).^{y}(1.^{y^{-1}}\alpha(y,z))=\alpha(x,y)\alpha(y,z).$
\end{center}
\par Now if $(a,x)\sim(b,y)$ and $(b,y)\sim(c,z)$, then
 \begin{center}
 $\tau(x)=\tau(y)=\tau(z)$, $a\alpha(x,y)=b$ and $b\alpha(y,z)=c$.
\end{center}
Hence, by $(3.1)$, $a\alpha(x,z)=a\alpha(x,y)\alpha(y,z)=b\alpha(y,z)=c.$ This implies that $\sim$ is transitive. By $(3.1)$, we have:
$$\alpha^{-1}(x,y)=\alpha(y,x). \eqno (3.2)$$
If $(a,x)\sim (b,y)$, then $\tau(x)=\tau(y)$ and $a\alpha(x,y)=b$. Thus, $a=b\alpha^{-1}(x,y)=b\alpha(y,x)$. Consequently, $\sim$ is an equivalence relation in $A\rtimes P$.
\par Now, Define $N_{\alpha}=\{(a,x)|(a,x)\in A\rtimes P, (a,x)\sim(1,1)\}$. Then, by the definition of $\sim$ and the identity $(3.2)$, we have
\\
$$N_{\alpha}=\{(a,x)|(a,x)\in A\rtimes P, \tau(x)=1, a\alpha(x,1)=1\}$$
$$=\{(a,x)|(a,x)\in A\rtimes P, \tau(x)=1, a=\alpha(1,x)\}.$$
Suppose that $(a,x), (b,x)\in N_{\alpha}$. Since $\tau(x)=1$, then $(a,x)(b,y)=(a^{x}b,xy)=(ab,xy)$. On the other hand, $\alpha(1,xy)=\alpha(1,x)\alpha(1,y)=ab$. Thus, $(a,x)(b,y)\in N_{\alpha}$. Also $(a,x)^{-1}=(^{x^{-1}}a^{-1},x^{-1})=(a^{-1},x^{-1})$ and $\alpha(1,x^{-1})=\alpha(xx^{-1},x^{-1})=\alpha(x,1)^{x}\alpha(x^{-1},x^{-1})=\alpha(x,1)=\alpha^{-1}(1,x)=a^{-1}$, hence $(a,x)^{-1}\in N_{\alpha}$. Therefore $N_{\alpha}$ is a subgroup of $A\rtimes P$.
\par Denote by $E_{\alpha}$ the quotient space $(A\rtimes P)/\sim$.
\\We show that $\sim$ is a congruence (i.e., $\sim$ is compatible with the group product) and therefore $E_{\alpha}$ is a group.
\par Let $(a,x)\sim(a',x')$ and $(b,y)=(b',y')$. Then $\tau(x)=\tau(x')$, $\tau(y)=\tau(y')$, $a\alpha(x,x')=a'$ and $ b\alpha(y,y')=b'$.
\par We have $(a,x)(b,y)=(a^{x}b,xy)$, $(a',x')(b',y')=(a'^{x'}b',x'y')$. Since $A$ is abelian, then
$$a^{x}b\alpha(xy,x'y')=a^{x}b\alpha(x,x')^{x}\alpha(y,y')=a\alpha(x,x')^{x}(b\alpha(y,y'))=a'^{x}b'=a'^{x'}b'.$$ Hence, $(a,x)(b,y)\sim(a',x')(b',y')$. Thus, $E_{\alpha}$ is a group. It is clear that $E_{\alpha}=(A\rtimes P)/N_{\alpha}$. This implies that $N_{\alpha}$ is a normal subgroup of $A\rtimes P$. Therefore $E_{\alpha}$ is a topological group.

\par There is a diagram as follows:
$$\xymatrix{M \ar [d]^{\alpha} \ar @{->} @< 3pt> [r]^{\iota_{0}}
\ar @{ ->} @<-3pt> [r]_{\iota_{1}} & P\ar [d]^{\beta}\ar [r]^{\tau}& G\ar @{=}[d] \ar [r]&1\\
A\ar [r]^{\sigma} & E_{\alpha} \ar[r]^{\psi} &G\ar[r] & 1} \eqno(**)$$
where $\sigma(a)=(a,1)N_{\alpha}$, $\psi(a,x)N_{\alpha}=\tau(x)$ and  $\beta(x)=(1,x)N_{\alpha}$.
\par It is obvious that $\sigma$ and $\beta$ are continuous homomorphisms. Consider the continuous map $\phi:A\rtimes P\rightarrow P, (a,x)\mapsto \tau(x)$.
\par We know that $\tau$ has a continuous section $s:G\rightarrow P$. Hence, $\beta\circ s$ is a continuous section for $\psi$, since
$\psi\beta\circ s=(\psi\beta)\circ s=\tau\circ s=Id_{G}$. Thus, by Lemma 1.6, $\phi$ is an open epimorphism.
\par If $\sigma(a)\in N_{\alpha}$, then $(a,1)\in N_{\alpha}$. Thus, by definition of $N_{\alpha}$, we get $a=\alpha(1,1)=1$. Therefore $\sigma$ is one to one.
\par $\psi\circ\sigma(a)=\psi((a,1)N_{\alpha})=\tau(1)=1$, thus $Im\sigma\subset Ker\psi$. On the other hand, if $\psi(a,x)N_{\alpha}=1$, then $\tau(x)=1$ and $\sigma(a\alpha(x,1))=(a,x)N_{\alpha}$, i.e., $Ker\psi\subset Im\sigma$. Therefore, $Im\sigma=Ker\psi$.
\par
Now, we prove that $\sigma$ is a homeomorphic embedding.
\\ Define the map $\chi:A\rtimes P\rightarrow A$ via $\chi(a,x)=a\alpha(x,s\tau(x)s(1)^{-1})$. Obviously, $\chi$ is continuous. It is clear that $(A\times \{1\})N_{\alpha}$ is a topological subgroup of $A\rtimes P$ and $(A\times \{1\})N_{\alpha}=\{(a,x)|a\in A, x\in Ker\tau\}$. Take $\hat{\chi}=\chi|_{(A\times \{1\})N_{\alpha}}$. Thus, $\hat{\chi}$ is continuous and $\hat{\chi}(a,x)=a\alpha(x,1)$. Note that $\hat{\chi}$ is a homomorphism, because
$$\hat{\chi}((a,x)(b,y))=\hat{\chi}(ab,xy)=ab\alpha(xy,1)=ab\alpha(x,1)\alpha(xy,x)$$
$$=ab\alpha(x,1)\alpha(x^{-1}xy,1)=ab\alpha(x,1)\alpha(y,1)=\hat{\chi}(a,x)\hat{\chi}(b,y).$$
If $(a,x)\in N_{\alpha}$, then $\hat{\chi}(a,x)=a\alpha(x,1)=\alpha(1,x)\alpha(x,1)=\alpha(1,1)=1$. Hence $\hat{\chi}(N_{\alpha})=1$. So, $\hat{\chi}$ induces the continuous homomorphism $$\xi:((A\times\{1\})N_{\alpha})/N_{\alpha}\rightarrow A, \xi((a,x)N_{\alpha})=\hat{\chi}(a,x).$$
It is clear that $Im\sigma=((A\times\{1\})N_{\alpha})/N_{\alpha}$. We have $\xi\circ\sigma(a)=a$ and
$$\sigma\circ\xi((a,x)N_{\alpha})=\sigma(a\alpha(x,1))=(a\alpha(x,1),1)N_{\alpha}=(a,x)(\alpha(x,1),x^{-1})N_{\alpha}$$
$$=(a,x)(\alpha(1,x^{-1}),x^{-1})N_{\alpha}=(a,x)N_{\alpha}.$$
Thus, $\sigma|^{Im\sigma}:A\rightarrow Im\sigma$, $\sigma|^{Im\sigma}(a)=(a,1)N_{\alpha}$, is a topological isomorphism. Therefore, $\sigma$ is a homeomorphic embedding.
\par
 Denote by $e_{\alpha}$ the extension $0\to A\stackrel{\sigma}\to E_{\alpha}\stackrel{\psi}\to G\to 1$. We have  $[e_{\alpha}]\in Opext_{s}(G,A)$. Since $s:G\rightarrow P$ is a continuous section for $\tau$, then $\beta\circ s$ is a continuous section for $\psi$ and
$$\beta\circ s(g)\sigma(a)\beta\circ s(g)^{-1}=(1,s(g))N_{\alpha}(a,1)N_{\alpha}(1,s(g)^{-1})N_{\alpha}$$ $$=(1,s(g))(a,1)(1,s(g)^{-1})N_{\alpha}=(^{s(g)}a,s(g))(1,s(g)^{-1})N_{\alpha}$$ $$=(^{s(g)}a,1)N_{\alpha}=(^{g}a,1)N_{\alpha}=\sigma(^{g}a).$$
In addition, $\sigma$ is one to one, thus the extension $e_{\alpha}$ induces the given action of $G$ on $A$. Hence $[e_{\alpha}]\in Opext_{s}(G,A)$.
\par Consider the map  $\zeta:Z^{1}(M,A)\rightarrow Opext_{s}(G,A)$ via $\alpha\mapsto [e_{\alpha}]$. Suppose that  $\alpha,\bar{\alpha}\in Z^{1}(M,A)$ and $\bar{\alpha}=\eta(\gamma)\alpha$, for some $\gamma\in Der_{c}(P,A)$. Define the map $\varepsilon:A\rtimes P\rightarrow A\rtimes P$ by $(a,x)\mapsto (a\gamma^{-1}(x),x)$. Obviously, $\varepsilon$ is continuous, and
$$\varepsilon((a,x)(b,y))=\varepsilon(a^{x}b,xy)=(a^{x}b\gamma^{-1}(xy),xy)=(a^{x}b^{x}\gamma^{-1}(y)\gamma^{-1}(x),xy)$$
$$=(a\gamma^{-1}(x)^{x}(b\gamma^{-1}(y)),xy)=(a\gamma^{-1}(x),x)(b\gamma^{-1}(y),y);$$
that is, $\varepsilon$ is a homomorphism.
\par If $(a,x)\in N_{\alpha}$, then
$$\bar{\alpha}(1,x)=(\gamma\iota_{0}.\gamma\iota_{1}^{-1}.\alpha)(1,x)=\gamma\iota_{0}(1,x)\gamma\iota_{1}^{-1}(1,x)\alpha(1,x)=\gamma^{-1}(x)\alpha(1,x)$$ $$=a\gamma^{-1}(x).$$
Hence, $(a\gamma^{-1}(x),x)\in N_{\bar{\alpha}}$, i.e., $\varepsilon(N_{\alpha})\subset N_{\bar{\alpha}}$. Thus, $\varepsilon$ induces the continuous homomorphism $\hat{\varepsilon}:E_{\alpha}\rightarrow E_{\bar{\alpha}}, (a,x)N_{\alpha}\mapsto (a\gamma^{-1}(x),x)N_{\bar{\alpha}}$.
\par Suppose that the extensions $e_{\alpha}:0\to A\stackrel{\sigma}\to E_{\alpha}\stackrel{\psi}\to G\to 1$ and $e_{\bar{\alpha}}:0\to A\stackrel{\bar\sigma}\to E_{\bar\alpha}\stackrel{\bar\psi}\to G\to 1$ are corresponding to $\alpha$ and $\bar{\alpha}$, respectively. We show that the following diagram commutes.
$$\xymatrix{e_{\alpha}:0\ar[r] &A\ar[r]^{\sigma}\ar @{=}[d]& E_{\alpha}\ar[r]^{\psi}\ar[d]^{\hat{\varepsilon}}&G\ar[r]\ar @{=}[d] & 1
\\ e_{\bar{\alpha}}:0\ar[r] &A\ar[r]^{\bar{\sigma}}& E_{\bar{\alpha}}\ar[r]^{\bar{\psi}}&G\ar[r] & 1}$$
Because, $\bar{\psi}\hat{\varepsilon}(a,x)N_{\alpha}=\bar{\psi}(a\gamma^{-1}(x),x)N_{\bar{\alpha}}=\tau(x)=\psi(a,x)N_{\alpha},$ and $\hat{\varepsilon}{\sigma}(a)=\hat{\varepsilon}(a,1)N_{\alpha}=(a\gamma^{-1}(1),1)N_{\bar{\alpha}}=(a,1)N_{\bar{\alpha}}=\bar{\sigma}(a).$ i.e., $[e_{\alpha}]=[ e_{\bar{\alpha}}]$. Thus, $\zeta$ induces the map $$\Theta:Coker\eta\rightarrow Opext_{s}(G,A)$$ $$[\alpha]\mapsto [e_{\alpha}].$$
We will prove that $\Theta$ is a bijective map.
\par
Let $e:\xymatrix{0\ar[r]& A\ar@{^(->}[r]^{\sigma} & C\ar[r]^{\psi}& G\ar[r]& 1}$ be an extension with a continuous section $s:G\rightarrow A$ and corresponding to the given way in which $G$ acts on $A$. Since $P$ is a projective topological group, then there is a continuous homomorphism $\beta$ so that
$$\xymatrix{P\ar [d]^{\beta}\ar [r]^{\tau}& G\ar @{=}[d] \ar [r]&1\\
 C \ar[r]^{\psi} &G\ar[r] & 1}$$
commutes. Now, define $\alpha_{\beta}:M\rightarrow A$ by $\alpha_{\beta}(m)=(\beta\iota_{0}\beta\iota_{1}^{-1})(m)$. Note that $\alpha_{\beta}$ is well-defined, because
$\psi(\beta\iota_{0}\beta\iota_{1}^{-1})=\psi\beta\iota_{0}\psi\beta\iota_{1}^{-1}=\tau\iota_{0}\tau\iota_{1}^{-1}=\mathbf{1}$. Thus, $Im\alpha_{\beta}\subset Ker\psi=Im\sigma=A$. For simplicity, we denote $\alpha_{\beta}$ by $\beta{\iota_{0}}\beta\iota^{-1}_{1}$. It is clear that $\alpha_{\beta}$ is continuous and in addition, $\alpha_{\beta}$ is  a crossed homomorphism, since
$$\alpha_{\beta}((x_{1},y_{1})(x_{2},y_{2}))=\alpha_{\beta}(x_{1}x_{2},y_{1}y_{2})=\beta(x_{1}x_{2})\beta^{-1}(y_{1}y_{2})$$
$$=\beta(x_{1})\beta(x_{2})\beta^{-1}(y_{2})\beta^{-1}(y_{1})=\beta(x_{1})\alpha_{\beta}(x_{2},y_{2})\beta^{-1}(y_{1})$$
$$=\alpha_{\beta}(x_{1},y_{1})\beta(y_{1})\alpha_{\beta}(x_{2},y_{2})\beta^{-1}(y_{1})=\alpha_{\beta}(x_{1},x_{2}).^{\beta(y_{1})}\alpha_{\beta}(x_{2},y_{2})$$
$$=\alpha_{\beta}(x_{1},y_{1}).^{\psi\beta(y_{1})}\alpha_{\beta}(x_{2},y_{2})=\alpha_{\beta}(x_{1},y_{1}).^{\tau(y_{1})}\alpha_{\beta}(x_{2},y_{2})$$
$$=\alpha_{\beta}(x_{1},y_{1}).^{(x_{1},y_{1})}\alpha_{\beta}(x_{2},y_{2}),$$
and $\alpha_{\beta}(x,x)=\beta(x)\beta^{-1}(x)=1$. i.e., $\alpha_{\beta}(\Delta_{P})=1$, hence $\alpha_{\beta}\in Z^{1}(M,A)$.
Consequently, $\beta$ and $\alpha_{\beta}$ make the following diagram commutative.
$$\xymatrix{& M \ar [d]^{\alpha_{\beta}} \ar @{->} @< 3pt> [r]^{\iota_{0}}
\ar @{ ->} @<-3pt> [r]_{\iota_{1}} & P\ar [d]^{\beta}\ar [r]^{\tau}& G\ar @{=}[d] \ar [r]&1\\
0\ar[r]&A\ar [r]^{\sigma} & C \ar[r]^{\psi} &G\ar[r] & 1}$$
\par Now, we will show that $[\alpha_{\beta}]$ is independent of the choice of $\beta:P\rightarrow C$.
\\ Consider the commutative diagram.
$$\xymatrix{& M \ar [d]^{\alpha_{\beta}}_{\alpha_{\bar{\beta}}} \ar @{->} @< 3pt> [r]^{\iota_{0}}
\ar @{ ->} @<-3pt> [r]_{\iota_{1}} & P\ar [d]^{\beta}_{\bar{\beta}}\ar [r]^{\tau}& G\ar @{=}[d] \ar [r]&1\\
0\ar[r]&A\ar [r]^{\sigma} & C \ar[r]^{\psi} &G\ar[r] & 1}$$
Define $\gamma(x)=(\bar{\beta}\beta^{-1})(x)$, for all $x\in P$. Since $\psi\gamma(x)=\psi\bar{\beta}(x)\psi\beta^{-1}(x)=\tau(x)\tau(x)^{-1}=1$, then $\gamma(x)\in Ker\psi=A$. Hence, we can define the continuous map $\gamma:P\rightarrow A$, $x\mapsto \gamma(x)$. Now $\gamma$ is a crossed homomorphism, since
$$\gamma(xy)=(\bar{\beta}\beta^{-1})(xy)=\bar{\beta}(xy)\beta^{-1}(xy)=\bar{\beta}(x)\bar{\beta}(y)\beta^{-1}(y)\beta^{-1}(x)$$
$$=\bar{\beta}(x)\gamma(y)\beta^{-1}(x)=\gamma(x)\beta(x)\gamma(y)\beta^{-1}(x)=\gamma(x)^{\beta(x)}\gamma(y)$$
$$=\gamma(x)^{\psi\beta(x)}\gamma(y)=\gamma(x)^{x}\gamma(y).$$
We have $\bar{\beta}=\gamma\beta$. Thus,
\begin{center}
$\alpha_{\bar{\beta}}=\bar{\beta}\iota_{0}\bar{\beta}\iota^{-1}_{1}=\gamma\beta\iota_{0}\gamma\beta\iota^{-1}_{0}=\gamma\iota_{0}\beta\iota_{0}\beta\iota^{-1}_{1}\gamma\iota^{-1}_{1}=\gamma\iota_{0}(\beta\iota_{0}\beta\iota^{-1}_{1})\gamma\iota^{-1}_{1}=\gamma\iota_{0}\alpha_{\beta}\gamma\iota^{-1}_{1}=\alpha_{\beta}(\gamma\iota_{0}\gamma_{1}^{-1}).$
\end{center}
i.e., $\alpha_{\bar{\beta}}\sim\alpha_{\beta}$.
\par Now, we will show that $\Theta([\alpha_{\beta}])=[e]$. It is sufficient to prove that  $e_{\alpha_{\beta}} \sim e$.
Define the map $\nu:A\rtimes P\rightarrow C$ via $(a,x)\mapsto a\beta(x)$. Obviously, $\nu$ is continuous, and since
\begin{center}
$\nu((a,x)(b,y))=\nu(a^{x}b,xy)=a^{x}b\beta(xy)=a\beta(x)^{\beta(x)^{-1}}(^{x}b)\beta(y)=\nu(a,x)^{\psi\beta(x)^{-1}}(^{x}b)\beta(y)=\nu(a,x)^{x^{-1}}(^{x}b)\beta(y)=\nu(a,x)\nu(b,y),$
\end{center}
then $\nu$ is a homomorphism. On the other hand, if $(a,x)\in N_{\alpha_{\beta}}$, then
$$\nu(a,x)=a\beta(x)=\alpha_{\beta}(1,x)\beta(x)=\beta^{-1}(x)\beta(x)=1.$$
Thus, $\nu:A\rtimes P\rightarrow C$ induces the continuous homomorphism $\hat{\nu}:E_{\alpha_{\beta}}\rightarrow C$, $(a,x)N_{\alpha_{\beta}}\mapsto a\beta(x)$. Note that $\hat{\nu}$ commutes the following diagram
$$\xymatrix{e_{\alpha_{\beta}}:0\ar[r] & A\ar[r]^{\bar{\sigma}} &E_{\alpha_{\beta}}\ar[r]^{\bar{\psi}} \ar[d]^{\hat{\nu}} & G\ar[r] & 1
\\e:0\ar[r] & A\ar@{^(->}[r]^{\sigma} \ar@{=}[u] & C\ar[r]^{\psi} & G\ar[r] \ar@{=}[u] & 1.}$$
Since $\hat{\nu}\bar{\sigma}(a)=\hat{\nu}((a,1)N_{\alpha_{\beta}})=a\beta(1)=a=\sigma(a)$, and $\psi\hat{\nu}((a,x)N_{\alpha_{\beta}})=\psi(a\beta(x))=\psi\beta(x)=\tau(x)=\bar{\psi}((a,x)N_{\alpha_{\beta}})$. Therefore, $e_{\alpha_{\beta}}\sim e$. i.e., $\Theta$ is onto.
\par We show that $\Theta$ is one to one.
\\Let $e_{i}:\xymatrix{0\ar[r]& A\ar@{^(->}[r]^{\sigma_{i}} & C_{i}\ar[r]^{\psi_{i}}& G\ar[r]& 1}$, $i=0, 1$, be the extensions with continuous sections and corresponding to the given way in which $G$ acts on $A$. Suppose that $e_{1}\sim e_{2}$ are equivalent, i.e.,  there is continuous homomorphism $\omega:C_{1}\rightarrow C_{2}$ such that  $\omega(a)=a, \forall a\in A$, and $\psi_{2}\omega=\psi_{1}$.   There is a continuous homomorphism $\beta:P\rightarrow C_{1}$ such that $\psi_{1}\beta=\tau$. Take $\bar{\beta}=\omega\circ\beta$. Consider the following diagram.
$$\xymatrix@!0{ & &M \ar@{->}[dddl]_{\alpha_{\bar{\beta}}}\ar@{->}[dd]^{\alpha_{\beta}} \ar @{->} @< 3pt> [rr]^{\iota_{0}}
\ar @{ ->} @<-3pt> [rr]_{\iota_{1}}
& & P \ar@{->}[dddl]_{\bar{\beta}}\ar@{->}[rr]^{\tau}\ar@{->}[dd]^{\beta} & & G\ar@{=}[dddl]\ar@{=}[dd]\ar[r] & 1
\\
& & & & & & &
\\
&0\ar@{-->}[r] &A \ar@{=}[dl] \ar@{^(-}[r]^{\ \ \sigma_{1}}
& \ar@{-->}[r]& C_{1}\ar@{->}[dl]^{\omega}\ar@{-}[r]^{\ \ \psi_{1}} & \ar@{-->}[r]& G\ar@{=}[dl]\ar[r] &1
\\
0\ar[r]& A \ar@{^(->}[rr]_{\sigma_{2}}
& & C_{2} \ar@{->}[rr]_{\psi_{2}} & &G\ar[r]  &1}$$
Since the  back, the bottom, the front squares, and  the middle triangle are commutative, so is the left triangle, i.e., $\alpha_{\bar{\beta}}=\alpha_{\beta}$. This shows that $\Theta$ is one to one. Consequently, $\Theta$ is bijective.
\par Finally, we prove that $\Theta$ is a homomorphism.
\\ Suppose that $\Theta([\alpha_{i}])=[e_{i}]$, $i=0, 1$. In another words, the following diagram is commutative.
$$\xymatrix{&M\ar @{->} @< 3pt>[r]^{\iota_{0}}\ar @{->} @< -3pt>[r]_{\iota_{1}} \ar[d]^{\alpha_{i}} & P\ar[r]^{\tau} \ar[d]^{\beta_{i}}& G\ar[r] \ar@{=}[d]&1
\\e_{i}:0\ar[r]&A\ar[r]^{\sigma_{i}}&E_{i}\ar[r]^{\psi_{i}}& G\ar[r]&1}$$
for $i=0, 1$.
Hence, we have the commutative diagram
$$\xymatrix{&M\ar @{->} @< 3pt>[r]^{\iota_{0}}\ar @{->} @< -3pt>[r]_{\iota_{1}} \ar[d]_{\alpha_{1}\times \alpha_{2}} & P\ar[r]^{\tau} \ar[d]_{\beta_{1}\times \beta_{2}}& G\ar[r] \ar[d]_{\mathbf{\Delta}_{G}}&1
\\e_{1}\times e_{2}:0\ar[r]&A\times A\ar[r]^{\sigma_{1}\times \sigma_{2}}&E_{1}\times E_{2}\ar[r]^{\psi_{1}\times \psi_{2}}& G\times G\ar[r]&1}$$
where $\mathbf{\Delta}_{G}(g)=(g,g)$, $\beta_{1}\times\beta_{2}(p)=(\beta_{1}(p),\beta_{2}(p))$, $\alpha_{1}\times\alpha_{2}(m)=(\alpha_{1}(m),\alpha_{2}(m))$, $\sigma_{1}\times\sigma_{2}(a,b)=(\sigma_{1}(a),\sigma_{2}(b))$  and $\psi_{1}\times\psi_{2}(e_{1},e_{2})=(\psi_{1}(e_{1}),\psi_{2}(e_{2}))$.
\\ Then by definition of pushout and pullback extensions [Al], we get the following diagram
$$\xymatrix{&
& M \ar@{->}@< 3pt>[rr]^{\iota_{0}}\ar@{->}@< -3pt>[rr]_{\iota_{1}}\ar@{->}'[d]^{\alpha_{1}\times\alpha_{2}}[dd]
& & P \ar@{->>}[rr]^{\tau}\ar@{->}'[d]^{\beta_{1}\times\beta_{2}}[dd] & & G\ar@{=}[dl]\ar@{->}^{\mathbf{\Delta}_{G}}[dd]
\\
e_{4}:& A \ar@{<-}[ur]^{\alpha_{1}.\alpha_{2}}\ \ar@{>->}[rr]\ar@{=}[dd]
& & E_{4} \ar@{<..}[ur]^{\gamma}\ar@{->}[dd] \ar@{->>}[rr]& & G\ar@{-}[d]^{\mathbf{\Delta}_{G}}
\\
& & A\times A \ \ar@{>->}'[r]^{\ \ \ \sigma_{1}\times\sigma_{2}}[rr]
& & E_{1}\times E_{2}\ar@{->>}'[r]^{\ \ \ \psi_{1}\times\psi_{2}}[rr] & & G\times G\ar@{=}[dl]
\\
e_{3}:& A \ \ar@{>->}[rr]\ar@{<-}[ur]^{\nabla_{A}}
& & E_{3} \ar@{<-}[ur]\ar@{->>}[rr] & &G\times G\ar@{<-}[uu]
}( ***)$$
in which, extension $e_{3}=\nabla_{A}(e_{1}\times e_{2}):\xymatrix{ 0\ar[r]&A\ar[r] & E_{3}\ar[r]& G\times G\ar[r]&1}$ is the pushout extension of $e_{1}\times e_{2}$, and extension $e_{4}=(\nabla_{A}(e_{1}\times e_{2}))\mathbf{\Delta}_{G}:\xymatrix{ 0\ar[r]&A\ar[r] & E_{4}\ar[r]& G\ar[r]&1}$ is the pullback extension of $\nabla_{A}(e_{1}\times e_{2})$. Also, $\nabla_{A}$ is  the codiagonal map, i.e.,  $\nabla_{A}(a,b)=ab$, and $\alpha_{1}.\alpha_{2}(m)=\alpha_{1}(m)\alpha_{2}(m)$.
Obviously,  the left and the right squares are commutative. Since $P$ is a projective topological group, then the  existence of $\gamma:P\rightarrow A$ is guaranteed and commutes the right-up square. Therefore in the right cube of diagram $(***)$, all the  up, down, front, back and right squares are commutative. So $\gamma$ commutes the middle square. On the other hand, in the left cube of diagram $(***)$, all the right, left, back, front and down squares are commutative. Thus, $\gamma$ commutes the left-up square of diagram $(***)$. Therefore, the whole top of diagram $(***)$ is commutative. This means that $\Theta([\alpha_{1}.\alpha_{2}])=[e_{4}]$. On the other hand by definition of Baer sum in the $Opext_{s}(G,A)$, $[\nabla_{A}(e_{1}\times_{2})\mathbf{\Delta}_{G}]=[e_{1}]+[e_{2}]$. Consequently, $\Theta(\alpha_{1}.\alpha_{2})=[e_{1}]+[e_{2}]=\Theta(\alpha_{1})+\Theta(\alpha_{2})$, i.e., $\Theta$ is an isomorphism and thereby $\Theta$ is an isomorphism. The proof is completed. 
\end{proof}
Note that Theorem \ref{Theorem 2.1.}  is similar to \cite[Proposition 6]{Ina-n} in a topological context.
\begin{cor} Let $P$ be a projective topological group, then $H^{2}(P,A)=0$, for every abelian topological $P$-module $A$. In paricular, for every Markov (Graev) free topological groups $F$, $H^{2}(F,A)=0$, for every abelian topological $F$-module $A$.
\end{cor}
\begin{proof} It is clear that $(\Delta_{P},\iota_{0},\iota_{G},Id_{P})$  is a standard projective simplicial kernel of $P$. Thus,  every continuous crossed homomorphism $\alpha:\Delta_{P}\to A$ is the zero map. Consequently, $H^{2}(P,A)=0$.
\end{proof}
Thus, we have discovered a general result of \cite[(5.6)]{Hu}.


\begin{thebibliography}{8}
\bibitem{Al} R.C. Alperin and H. Sahleh, Hopf's formula and the Schur multiplicator for topological groups, Kyungpook. Math. Journal, 31(1) (1991), 35-71.
\bibitem{Gra} M. I. Graev, Free topological groups, (Russian) Izv. Akad. Nauk SSSR, Ser. Mat. 12 (1948), 279-324.
\bibitem{Hal} C. E. Hall, $\mathfrak{F}$-Projective Objects, Proc. Amer. Math. Soc. 26, (1970), 193-195.
\bibitem{Hu} S.T. Hu, Cohomology theory in topological groups, Michigan Math. J., 1(1) (1952), 11-59.
\bibitem{Ina-n} H. Inassaridze, Non-abelian cohomology of groups, Georgian Math. J. 4, No 4 (1997), 313-332.
\bibitem{Mar} A. A. Markov, On free topological groups. (Russian) Izv. Akad. Nauk SSSR, Ser. Mat. 9 (1945), 3-64.
\bibitem{Sah1} H. Sahleh, H.E. Koshkoshi,  First Non-Abelian Cohomology Of Topological Groups. Gen. Math. Notes, 2012, 12(1), 20-34.
\bibitem{Tho} B.V.S. Thomas, Free topological groups, General Topology and its Appl., 1 (1971), pp. 51–72.
\end{thebibliography}
\end{document}